\documentclass{alggeom}
\usepackage{amsmath}
\usepackage{MnSymbol}
\usepackage{ifthen}
\newcommand{\cit}[1]{{\rm \textbf{#1}}}

\newcommand{\Ref}[2]{\cit{%
\ifthenelse{\equal{#1}{thm}}{Theorem}{}%
\ifthenelse{\equal{#1}{prop}}{Proposition}{}%
\ifthenelse{\equal{#1}{lem}}{Lemma}{}%
\ifthenelse{\equal{#1}{cor}}{Corollary}{}%
\ifthenelse{\equal{#1}{defn}}{Definition}{}%
\ifthenelse{\equal{#1}{oss}}{Remark}{}%
\ifthenelse{\equal{#1}{rmk}}{Remark}{}%
\ifthenelse{\equal{#1}{sec}}{Section}{}%
\ifthenelse{\equal{#1}{ex}}{Example}{}%
\ifthenelse{\equal{#1}{conj}}{Conjecture}{}%
\ifthenelse{\equal{#1}{ssec}}{Subsection}{}%
\ifthenelse{\equal{#1}{tab}}{Table}{}%
\ifthenelse{\equal{#1}{cla}}{Claim}{}%
\  \ref{#1:#2}%
}}
\usepackage{hyperref}

\newtheorem{prop}{Proposition}[section]
\newtheorem*{prop*}{Proposition}
\newtheorem{thm}[prop]{Theorem}
\newtheorem*{thm*}{Theorem}
\newtheorem{lem}[prop]{Lemma} 
\newtheorem{cor}[prop]{Corollary}
\newtheorem*{cor*}{Corollary}
\newtheorem{conj}[prop]{Conjecture}

\theoremstyle{definition}
\newtheorem{defn}[prop]{Definition}
\newtheorem{ex}[prop]{Example}

\theoremstyle{remark}
\newtheorem{oss}[prop]{Remark}
\newtheorem{rmk}[prop]{Remark}

\numberwithin{equation}{section}
\newcommand{\hk}{hyperk\"{a}hler }

\newcommand{\kahl}{K\"{a}hler }
\newcommand{\ktiposp}{$K3^{[2]}$-type }
\newcommand{\ktipo}{$K3^{[2]}$-type}
\newcommand{\kntipo}{$K3^{[n]}$-type}
\newcommand{\kntiposp}{$K3^{[n]}$-type }
\newcommand{\is}{irreducible symplectic}
\newcommand{\issp}{irreducible symplectic }

\newcommand{\ie}{i.e.\ }
\DeclareMathOperator*{\rk}{rk}
\newcommand{\Pic}{\mathrm{Pic}}
\DeclareMathOperator{\Ext}{Ext}

\DeclareMathOperator{\ch}{ch}

\newcommand{\id}{\mathrm{id}}

\newcommand{\mc}[1]{\mathcal{#1}}
\DeclareMathOperator*{\NS}{NS}
\newcommand{\Z}{\mathbb{Z}}

\newcommand{\PP}{\mathbb{P}}
\let\div\undefined
\DeclareMathOperator{\div}{div}
\DeclareMathOperator{\td}{td}
\newcommand{\aut}{\mathrm{Aut}}
\begin{document}

\title{Induced Automorphisms on Irreducible Symplectic Manifolds}
\author{Giovanni Mongardi}
\email{giovanni.mongardi@unimi.it}
\address{Department of Mathematics, University of Milan\\ via Cesare Saldini 50, Milan}
\author{Malte Wandel}
\email{wandel@kurims.kyoto-u.ac.jp}
\address{Research Institute for Mathematical Sciences,
Kyoto University, Kitashirakawa-Oiwakecho, Sakyo-ku, Kyoto, 606-8502 Japan}
%
\classification{Primary: 14J50 secondary: 14D06, 14F05 and 14K30}
\keywords{irreducible symplectic manifolds, automorphisms, moduli spaces of stable objects}
\thanks{The first named author was supported by FIRB 2012 ``Spazi di Moduli e applicazioni'', by SFB/TR 45 ``Periods, moduli spaces and arithmetic of algebraic varieties'' and partially by the Max Planck Institute in Mathematics.\\
The second named author was supported by JSPS Grant-in-Aid for Scientific Research (S)25220701 and by the DFG research training group GRK 1463 (Analysis, Geometry and String Theory).}

\begin{abstract}
We introduce the notion of induced automorphisms in order to state a criterion to determine whether a given automorphism on a manifold of \kntiposp is, in fact, induced by an automorphism of a $K3$ surface and the manifold is a moduli space of stable objects on the $K3$. This criterion is applied to the classification of non-symplectic prime order automorphisms on manifolds of \ktiposp and we prove that almost all cases are covered. Variations of this notion and the above criterion are introduced and discussed for the other known deformation types of \issp manifolds. Furthermore we provide a description of the Picard lattice of several \issp manifolds having a lagrangian fibration.
\end{abstract}

\maketitle
\addcontentsline{toc}{section}{Abstract}
\vspace*{6pt}
\section*{Introduction}\label{sec:introduction}
\addcontentsline{toc}{section}{Introduction}
The present paper deals with several questions concerning irreducible holomorphic symplectic manifolds and their automorphisms. All known examples of \issp manifolds arise from symplectic surfaces, often as moduli spaces of sheaves on these surfaces. The easiest such example is the Hilbert scheme of $n$ points on a $K3$ surface, constructed by Beauville in \cite{beau3}. This kind of construction allows to produce several examples of automorphisms on \issp manifolds, simply by taking a $K3$ surface with non-trivial automorphism group and considering the induced action on its Hilbert scheme. These kinds of automorphisms are called \emph{natural} and were studied by Beauville \cite{Beau83b}, Boissi\`{e}re \cite{Boi12} and many others. Very few examples of non-natural automorphisms are known, such as those constructed in \cite{OW13}, and a numerical criterion to distinguish between natural and non-natural automorphisms is available only in special cases, see \cite{BS12} and \cite{Mon13}.

The main purpose of this paper is to provide a generalisation of the notion of natural automorphisms. This notion appeared the first time for moduli spaces of sheaves in the paper \cite{OW13}, a work inspired by the beautiful construction in \cite[Sect.\ 5]{OS11}. We extend the ideas drastically using recent developments in the theory of stability conditions by Bridgeland \cite{Bri08} by Bayer and Macr\`{i} (\cite{BM12} and \cite{BM13}) and Yoshioka \cite{Yos12}. This new notion of \emph{induced} automorphisms includes all automorphisms on moduli spaces of stable objects, where the action is induced from an automorphism of the underlying surface. We provide a numerical criterion for the recognition of such examples and we apply this general framework to construct several new examples of induced, yet non-natural automorphisms.
In particular, we prove the following:
\begin{thm*}
Let $X$ be a manifold of \kntipo, and let $G\subset \aut(X)$ be a group of non-symplectic automorphisms. Assume that the action of $G$ fixes a copy of $U$ inside $H^2(X)$, then there exists a $K3$ surface $S$ such that $G\subset \aut(S)$, $X$ is a moduli space of stable objects on $S$ and the action of $G$ on $X$ is induced by that on $S$. 
\end{thm*}
See \Ref{thm}{k3n_induced} for the proof and a more precise statement.
The same technique can be applied also to symplectic automorphisms, but in this case we obtain a more general result:
\begin{thm*}
The moduli space of pairs $(X,G)$ with $X$ \is, $G\subset \aut(X)$ symplectic and $G\subset O(H^2(X))$ fixed up to conjugation has at most the same number of connected components as the moduli space of marked pairs $(X,f)$.
\end{thm*}
See \Ref{thm}{sympl_def} for details.

We mainly apply the new technique to automorphisms which do not preserve the symplectic form (\ie non-symplectic automorphisms): In the recent work of Boissi\`{e}re, Camere and Sarti (\cite{BCS14}) a lattice-theoretic classification of non-symplectic automorphisms of prime order (different from $5$) is given in the case of manifolds of \ktipo. The case of involutions is of special interest since it provides a great number of automorphisms of \issp manifolds, where, up to now, no geometric realisation is known. Using the notion of induced automorphisms we can realise almost all the unknown examples as automorphisms on moduli spaces of sheaves. Some examples are provided using degenerations of double EPW-sextics. We therefore have the following:
\begin{cor*}
Aside from $3$ cases, there is a geometric realisation for every family of non-symplectic prime order ($\neq5$, $\leq 19$) automorphisms on manifolds of \ktipo. Almost all families correspond to induced automorphisms. 
\end{cor*} 

A few more results concerning automorphisms are included, in particular we have the following for lagrangian fibrations with a section.
\begin{prop*}
Let $X$ be a manifold of \kntiposp having a lagrangian fibration with a section. Then $\Pic(X)$ admits a primitive embedding of $U(2)$ if $n$ is odd, and of $(2)\oplus (-2)$ otherwise.
\end{prop*}
See \Ref{prop}{lagr_lattice} for the proof and the following statements for analogous results on other \issp manifolds. In particular we can exploit the above to prove that there exist lagrangian fibrations which can not be deformed to fibrations having a section.\\

The structure of the paper is as follows: In \Ref{sec}{prel}, we gather all preliminaries concerning lattice theory, \issp manifolds and stability conditions. In \Ref{sec}{periods_mod_spac} we give a numerical criterion to recognise a moduli space of stable objects on a $K3$ (which was also recently proved in \cite{Add14}) or abelian surface. In \Ref{sec}{induced_auto_group} we provide the general framework to produce and distinguish groups of induced automorphisms on manifolds of \kntipo. A conjectural setting for generalised Kummer manifolds is provided, which is proven to hold in the case that $n+1$ is a prime power. The case of symplectic automorphisms is treated in its full generality. In \Ref{sec}{kieran_case}, we discuss the notion of induced automorphisms on the two manifolds introduced by O'Grady. Finally, in \Ref{sec}{appl} we apply the theoretical construction to the recent work of Boissi\`{e}re, Camere and Sarti \cite{BCS14} to explicitly construct all non-symplectic automorphisms of prime order (different from $5$) on manifolds of \ktipo, apart from three cases. Moreover, we analyse lagrangian fibrations with a section together with the involution which is naturally induced by this section. 

In this article all \issp manifolds are assumed to be projective, unless stated otherwise.

\section*{Acknowledgements}
\addcontentsline{toc}{section}{Acknowledgements}
We are grateful to Samuel Boissi\`{e}re, Chiara Camere and Alessandra Sarti for their comments and for letting us know about their work \cite{BCS14}. We would also like to thank Paolo Stellari for telling us about \cite{Yos12}.
The first named author would like to thank Matthias Sch\"{u}tt and the Institute for Algebraic Geometry of the University of Hannover for their kind hospitality and for providing a nice working environment where this work was started. Moreover he would also like to thank Arvid Perego and Antonio Rapagnetta for useful discussions.  
The second named author wants to thank Gilberto Bini for his kind hospitality, the latter and Matthias Sch\"{u}tt and the Vigoni exchange program for supporting his visit to Milano.
Finally, we are both grateful to the Max Planck Institut f\"ur Mathematik Bonn, for having partially supported the first named author and hosted the second named author.  

\section{Preliminaries}\label{sec:prel}
\subsection{Lattice theory}
Let $L$ be an even lattice, the group $A_L=L^\vee/L$ is called the \em discriminant group \em and the quadratic form of $L$ induces a form $q_{A_L}$ with values in $\mathbb{Q}/2\mathbb{Z}$. The \em length of $A_L$ \em is denoted by $l(A_L)$. If $X$ is an \issp manifold, we denote by $A_X$ the discriminant group of the lattice $H^2(X,\mathbb{Z})$. An \em overlattice \em of $L$ is any lattice $M\supset L$ such that $M/L$ is torsion. An embedding $L\hookrightarrow M$ is \em primitive \em if the quotient $M/L$ has no torsion. The \em divisor of $v\in L$, \em denoted $div(v)$, is the positive integer $n$ such that $(v,L)=n\mathbb{Z}$.

If $L$ is a lattice, we denote with $L(n)$ a lattice with the same structure as $\mathbb{Z}$ module but with quadratic form multiplied by $n$. We denote with $A_n$, $D_n$ and $E_n$ the positive definite lattices associated to the corresponding Dynkyn diagrams. Let $n$ be an integer, then we denote with $\langle n\rangle$ a rank 1 lattice generated by an element of square $n$ and with $\Lambda_{4n}$ the unique even unimodular lattice of signature $(4,4n-4)$.
\begin{oss}\label{oss:overlattice}
Overlattices of $L$ are in bijective correspondence with isotropic subgroups $H$ of $A_L$. Moreover the discriminant group of an overlattice given by $H$ is $H^\perp/H$
\end{oss}
\begin{lem}\cite[Lemma 3.5]{GHS10}
Let $L'$ be a lattice and let $L=U^2\oplus L'$. Let $v,w\in L$ be two elements such that the following holds:
\begin{itemize}\renewcommand{\labelitemi}{$\bullet$}
\item $v^2=w^2$.
\item $[v/div(v)]=[w/div(w)]$ in $A_L$.
\end{itemize}
Then there exists an isometry $g$ of $L$ such that $g(v)=g(w)$.
\end{lem}


\begin{cor}
Let $L=\Lambda_{24}$ or $\Lambda_{8}$. Then primitive embeddings of any corank $1$ lattice $M$ are determined, up to isometry, by the square of a generator of $M^\perp$.
\end{cor}

\begin{lem}\label{lem:aut_extend}
Let $M$ be a lattice and let $G\subset O(M)$. Let $M\hookrightarrow L$ be a primitive embedding in an unimodular lattice. Let $G$ act trivially on $A_M$. Then $G$ extends to a group of isometries of $L$ acting trivially on $M^\perp$.
\end{lem}

\begin{defn}
Let $L$ be a lattice and $G\subset O(L)$. We denote by $T_G(L)$ ($S_G(L)$) the \em invariant \em (the \em co-invariant) lattice of $L$. \em If $G$ acts on a manifold $X$, we denote by $T_G(X)$ ($S_G(X)$) the \em invariant \em (the \em co-invariant) lattice of the induced action on $H^2(X,\Z)$. \em
\end{defn}

\begin{oss}\label{oss:G_tors}
Let $L$ be a unimodular lattice and let $G\subset O(L)$. Then $A_{T_G(L)}$ is of $|G|$ torsion.
\end{oss}
\begin{defn}
A lattice $M$ is called \em $2$-elementary \em if $A_M=(\mathbb{Z}/2\Z)^a$. For such lattices, the invariant $\delta$ is defined to have value $0$ if the discriminant quadratic form is integer valued, $1$ otherwise.
\end{defn}
\begin{thm}\cite[Theorem 3.6.2]{Nik80}
A $2$-elementary indefinite lattice is uniquely determined by its rank, its signature, the length of its discriminant group and $\delta$
\end{thm}

Some basic building blocks for such lattices are $U$, $U(2)$, $D_4$, $E_8$ and $E_8(2)$ for $\delta=0$ and $\langle2\rangle$, $\langle-2\rangle$ for $\delta=1$.

\subsection{Useful notions on \issp manifolds}

Here we gather several known results on \issp manifolds. Many of these results are taken from the survey of Huybrechts \cite{Huy01}.

\begin{defn}
A \kahl manifold $X$ is called an \em irreducible holomorphic symplectic manifold \em (short: an \em \issp manifold\em) if the following hold:
\begin{itemize}\renewcommand{\labelitemi}{$\bullet$}
\item $X$ is compact.
\item $X$ is simply connected.
\item $H^{2,0}(X)=\mathbb{C}\sigma_X$, where $\sigma_X$ is an everywhere non-degenerate symplectic 2-form.
\end{itemize}
\end{defn}

There is the equivalent, but more differential geometric notion of \em \hk manifold. \em In this article we keep with \issp manifolds and refer the interested reader to the literature for a comparison of the two notions.

There are not many known examples of \issp manifolds and for a long time the only known ones were $K3$ surfaces (which are the only examples in dimension 2) and two families of examples given by Beauville \cite{beau3}:

\begin{ex}\label{ex:kntipo}
Let $S$ be a $K3$ surface and let $S^{(n)}$ be its $n$-th symmetric product. There exists a minimal resolution of singularities (called \em the Hilbert$-$Chow morphism\em)
\begin{equation}\nonumber
S^{[n]}\,\stackrel{HC}{\rightarrow}\,S^{(n)},
\end{equation}
where $S^{[n]}$ is the Douady space parametrizing zero dimensional analytic subsets of $S$ of length $n$. Furthermore this resolution of singularities endows $S^{[n]}$ with a symplectic form induced by the symplectic form on $S$. The manifold $S^{[n]}$ is an \issp manifold of dimension $2n$ and if $n\geq2$, we have $b_2(S^{[n]})=23$.\\  Whenever $X$ is an \issp manifold deformation equivalent to one of these manifolds, we will call $X$ \em a manifold of $K3^{[n]}$-type.\em
\end{ex}

\begin{ex}\label{ex:gen_kum_tipo}
Let $T$ be a complex $2$-torus and let
\begin{equation}\nonumber
T^{[n+1]}\,\stackrel{HC}{\rightarrow}\,T^{(n+1)}
\end{equation}
be the minimal resolution of singularities of the symmetric product as in \Ref{ex}{kntipo}. Beauville proved that the symplectic form on $T$ induces a symplectic form on $T^{[n+1]}$. However, this manifold is not \issp since it is not simply connected. But if we consider the summation map
\begin{align*}
T^{(n+1)} \stackrel{\Sigma}{\rightarrow}& T\\\nonumber
(t_1,\dots,t_{n+1}) \rightarrow & \sum_i t_i
\end{align*}
and set $K_n(T)=(\Sigma\circ HC)^{-1}(0)$, we obtain a new \issp manifold of dimension $2n$ called \em generalised Kummer manifold of $T$. \em If $n=1$ then $K_n(T)$ is just the usual Kummer surface, otherwise it has $b_2=7$. Note that the summation map $\Sigma\circ HC\colon T^{[n+1]}\rightarrow T$ is, in fact, the Albanese map of $T^{[n+1]}$.\\
Whenever $X$ is an \issp manifold deformation equivalent to one of these manifolds we will call $X$ \em a manifold of Kummer $n$-type. \em
\end{ex}

Two more deformation types of \issp manifolds are known and they were both discovered by O'Grady (see \cite{OGr99} and \cite{OGr03}). They are obtained as a symplectic resolution of singular moduli spaces of sheaves on $K3$ or abelian surfaces. We will denote by $Og_6$ the six-dimensional example and $Og_{10}$ the ten-dimensional example. It is known that $b_2(Og_6)=8$ and $b_2(Og_{10})=24$. We call manifolds which are deformation equivalent to $Og_{10}$ ($Og_6$) \em manifolds of $Og_{10}$-type ($Og_6$-type).\em

\begin{thm}
Let $X$ be an \issp manifold of dimension $2n$. Then there exists a canonically defined pairing $(\,,\,)_X$ on $H^2(X,\mathbb{C})$, the Beauville-Bogomolov pairing, and a constant $c_X$ (the Fujiki constant) such that the following holds:
\begin{equation*}
(\alpha,\alpha)_X^n=c_X\int_{X}\alpha^{2n}.
\end{equation*}

Moreover $c_X$ and $(\,,\,)_X$ are deformation and birational invariants.
\end{thm}

The Beauville-Bogomolov forms of known manifolds are the following:
\renewcommand{\arraystretch}{1.5}
\begin{table}[h]
\begin{tabular}{|c|c|}
\hline
Deformation class & Lattice structure on $H^2$\\
\hline
\kntiposp & $U\oplus U\oplus U\oplus E_8(-1)\oplus E_8(-1)\oplus (2-2n)$\\
\hline
Kummer $n$-type & $U\oplus U\oplus U\oplus (-2-2n)$\\
\hline
$Og_6$-type & $U\oplus U\oplus U\oplus (-2)\oplus (-2)$\\
\hline
$Og_{10}$-type & $U\oplus U\oplus U\oplus E_8(-1)\oplus E_8(-1)\oplus A_2(-1)$\\
\hline
\end{tabular}
\end{table}

Let $X$ be an \issp manifold, let us consider the natural map $\nu\colon\aut(X)\rightarrow O(H^2(X))$ associating to a morphism its induced action on cohomology. Hassett and Tschinkel \cite[Thm.\ 2.1]{HT13} prove that the kernel of the map $\nu$ is a deformation invariant of $X$. Some examples of such kernels are known: it is trivial if $X$ is the Hilbert scheme of points of a very general $K3$ (see \cite[Prop.\ 10]{Beau83b} and it is generated by the group of points of order $n+1$ and the sign change on an abelian variety $A$ if $X=K_n(A)$  (see \cite[Prop.\ 9]{Beau83b}).


\subsection{Moduli of \issp manifolds and the Torelli problem}
\begin{lem}
Let $X$ be an \issp manifold with \kahl class $\omega$ and symplectic form $\sigma_X$. Then there exists a family
\renewcommand{\arraystretch}{1.2}
\[\begin{array}{rcc} 
TW_{\omega}(X)& := &X\times\mathbb{P}^1\\\nonumber
 && \downarrow\\\nonumber
\{(a,b,c)\,\in\,\mathbb{R}^3,\,a^2+b^2+c^2=1\}=S^2 &\cong&\mathbb{P}^1 
\end{array}\]
called \emph{Twistor space} such that $TW_{\omega}(X)_{(a,b,c)}\cong X$ with complex structure given by the \kahl class $a\omega+b(\sigma_X+\overline{\sigma}_X)+c(\sigma_X-\overline{\sigma}_X)$.
\end{lem}

\begin{defn}
Let $X$ be an \issp manifold and let $H^2(X,\mathbb{Z})\,\cong\,N$. An isometry $f\colon H^2(X,\mathbb{Z})\,\rightarrow\,N$ is called \em a marking of $X$. \em A pair $(X,f)$ is called \em a marked \issp manifold.\em
\end{defn}

\begin{defn}
Let $(X,f)$ be a marked \issp manifold and let $H^2(X,\Z)\cong N$. Let $\mc{M}_N$ be the set $\{(X,f)\}/\sim$ of marked \issp manifolds,  where $(X,f)\sim (X',f')$ if and only if there exists an isomorphism $\phi\colon X\,\rightarrow\,X'$ such that $\phi^*=f^{-1}\circ f'$.
\end{defn}

\begin{defn}
Let $X$ be an \issp manifold and let $N$ be a lattice such that $H^2(X,\mathbb{Z})\cong N$. Then we define the \em period domain $\Omega_N$ \em as
\begin{equation*}
\Omega_N=\{x\in \mathbb{P}(N\otimes\mathbb{C})\,|\,(x,x)_N=0,\,(x+\overline{x},x+\overline{x})_N>0\}.
\end{equation*}
\end{defn}

\begin{defn}
Let $\mathcal{X}\rightarrow S$ be a flat family of deformations of $X$ and let $f$ be a marking of $X$ into the lattice $N$. Let moreover $F$ be a marking of $\mathcal{X}$ compatible with $f$. Then the \em period map \em $\mathcal{P}\colon S\,\rightarrow\,\Omega_N$ is defined as follows:
\begin{equation*}
\mathcal{P}(s):=F_s(H^{2,0}(\mathcal{X}_s)).
\end{equation*}
\end{defn}

The period map $\mc{P}$ of the (flat) familiy $\mathcal{X}\rightarrow Def(X)$ of deformations of $X$ is called \em the local period map.\em

\begin{thm}[Local Torelli, Beauville \cite{beau3}]\label{thm:local_torelli}
Let $(X,f)$ and $N$ be as above and let moreover $F$ be a compatible marking of $\mathcal{X}\rightarrow Def(X)$. Then the map $Def(X)\stackrel{\mathcal{P}}{\rightarrow} \Omega_N$ is a local isomorphism.
\end{thm}

This local isomorphism allows us to glue the various universal deformations giving $\mc{M}_N$ the structure of a complex space.
Another well known fact about the period map is the following:

\begin{thm}[Huybrechts, \cite{Huy01}]\label{thm:surj_period}
Let $\mc{M}_N^0$ be a connected component of $\mc{M}_N$. Then the period map $\mc{P}\colon\mc{M}_N^0\,\rightarrow\,\Omega_N$ is surjective.
\end{thm}

A weaker Global Torelli theorem holds, see \cite{Huy10}, \cite{Mar11} and \cite{Ver09}.

\begin{thm}[Global Torelli, Huybrechts, Markman and Verbitsky]\label{thm:global_torelli}
Let $X$ and $Y$ be two \issp manifolds. Suppose $\psi\colon H^2(X,\Z)\,\rightarrow\,H^2(Y,\Z)$ is a parallel transport operator preserving the Hodge structure. Then there exists a birational map $\phi\colon X\,\dashrightarrow\,Y$.
\end{thm}

A characterization of parallel transport operators is needed to determine birational manifolds. In the case of manifolds of \kntipo, this has been provided by Markman \cite{Mar11}. Let $L$ be a lattice isometric to $H^2(K3^{[n]},\mathbb{Z})$ and let $O(L,\Lambda_{24})$ denote the set of primitive embeddings of $L$ into $\Lambda_{24}$. Finally let $O(L,\Lambda_{24})/O(\Lambda_{24})$ be the orbit space of isometric primitive embeddings.

\begin{thm}
Let $X$ be a manifold of \kntipo. Then there exists a canonically defined equivalence class of an embedding $\iota_X\colon H^2(X,\mathbb{Z})\rightarrow \Lambda_{24}$. A Hodge isometry $g\colon H^2(X,\mathbb{Z})\rightarrow H^2(Y,\mathbb{Z})$ is a parallel transport operator if and only if $\iota_X=g\circ\iota_Y$ in $O(L,\Lambda_{24})/O(\Lambda_{24})$
\end{thm}

Note that the order of $O(L,\Lambda_{24})/O(\Lambda_{24})$ is $2^{r}$, where $r$ is the number of different prime factors of $n-1$.

Less is known about parallel transport for generalised Kummer manifolds, however there are some results due to Markman which appeared in \cite[Corollary 4.8]{MM12}. Let $X$ be a manifold of Kummer $n$-type and let $\mathcal{W}(X)$ be the group of orientation preserving isometries of $H^2(X,\mathbb{Z})$ acting as $\pm1$ on $A_X$. Let $\mathcal{N}(X)$ be the kernel of the map $\mathrm{det}\circ\chi\colon\mathcal{W}(X)\rightarrow {\pm 1}$, where $\chi$ is the character of the action on $A_X$. 
\begin{prop}\label{prop:mono_kum}
 Keep notation as above, then $Mon^2(X)\cap \mathcal{W}(X)=\mathcal{N}(X)$ and $Mon^2(X)=\mathcal{N}(X)$ if $n+1$ is a prime power.
\end{prop}
This implies that, if $n+1$ is a prime power, the moduli space of marked generalised kummer manifolds has $4$ connected components (corresponding to the four connected components of marked abelian surfaces).

\subsection{Double EPW-Sextics}\label{ssec:epw}
Double EPW-sextics were first introduced by O'Grady in \cite{OGr06}, they are in many ways a higher dimensional analogue to $K3$ surfaces obtained as the double cover of $\mathbb{P}^2$ ramified along a sextic curve.

Let $V\cong\mathbb{C}^6$ be a six dimensional vector space with basis given by $\{e_0,e_1,e_2,e_3,e_4,e_5\}$ and let 
\begin{equation}\nonumber
vol(e_0\wedge e_1\wedge e_2\wedge e_3\wedge e_4\wedge e_5)=1
\end{equation}
be a volume form, giving a symplectic form $\sigma$ on $\Lambda^3V$ defined by
\begin{equation}\nonumber
\sigma(\alpha,\beta)=vol(\alpha\,\wedge\,\beta).
\end{equation}

Let $\mathbb{LG}(\Lambda^3V)$ be the set of lagrangian subspaces of $\Lambda^3V$ with respect to $\sigma$. 
Furthermore let $F$ be the vector bundle on $\mathbb{P}(V)$ with fibre 
\begin{equation}\nonumber
F_{v}=\{\alpha\,\in\,\Lambda^3V\,,\,\alpha\,\wedge\,v=0\}.
\end{equation}

Let $A\in\mathbb{LG}(\Lambda^3V)$ and let $\lambda_A(v)$ be the following composition
\begin{equation}\label{EPW_eq}\nonumber
F_{v}\,\rightarrow\,\Lambda^3 V\rightarrow (\Lambda^3 V)/A,
\end{equation}
where the first map is the injection of $F_{v}$ as a subspace of $\Lambda^3 V$ and the second is the projection to the quotient with respect to $A$. Therefore we define $Y_A[i]$ as the following locus:
\begin{equation}\nonumber
Y_A[i]=\{[v]\,\in\,\mathbb{P}(V)\,,\,\dim(A\,\cap\,F_v)\geq i\}.
\end{equation}
Here $Y_A[1]=Y_A$ is the EPW-sextic associated to $A$ and coincides with the degeneracy locus of $\lambda_A$ if $A$ is general.

\begin{defn}
Let $\mathbb{LG}(\Lambda^3 V)^0$ be the open subset of lagrangian subspaces $A$ such that the following hold
\begin{itemize}\renewcommand{\labelitemi}{$\bullet$}
\item $Y_A[3]=\emptyset$,
\item $Gr(3,6)\cap \mathbb{P}(A)=\emptyset$, where $Gr(3,6)\,\subset\,\mathbb{P}(\Lambda^3 \mathbb{C}^6)$ via the Pl\"{u}cker embedding.
\end{itemize}
\end{defn}
Let us remark that $\mathbb{LG}(\Lambda^3V)^0$ contains the generic lagrangian subspace.

\begin{thm}\cite[Theorem 1.1]{OGr06}\label{thm:kieran_epw}
For $A\in\,\mathbb{LG}(\Lambda^3V)^0$ there exists a double cover $X_A\,\rightarrow\,Y_A$  ramified along $Y_A[2]$ such that $X_A$ is a manifold of \ktipo.
\end{thm}
A polarization $h$ of a Double EPW-sextic $X_A$ is given by the pullback of the hyperplane section $\mathcal{O}_{Y_A}(1)$ of $Y_A$, a direct computation yields $h^2=2$.

\begin{oss}
Let us look a little into what can happen if the lagrangian $A$ contains some decomposable tensors. If $A$ contains an isolated decomposable tensor $l$, then the double EPW-sextic $X_A$ is singular along a $K3$ surface obtained as the double cover of the plane $P$ associated to $l$, ramified along the sextic $P\cap Y_A[2]$. There exists a symplectic resolution of singularities $\overline{X}_A \rightarrow X_A$ obtained by blowing up the singular $K3$ surface. Moreover, as proven in \cite[Claim 3.8]{OGr12}, the class of this exceptional divisor is orthogonal to the (now semiample) divisor of square 2 coming from the polarization of $Y_A$ and has square $-2$ and divisor $1$.
\end{oss}

We will be particularly interested in the following construction made by Ferretti \cite[Prop. 4.3]{Fer12}:
\begin{prop}\label{prop:fer}
Let $S_0\subset \mathbb{P}^3$ be a generic quartic with $k\leq 10$ nodes and no other singularities. Then there exists a smooth complex contractible space $U$ of dimension $20-k$ with a distinguished point $0$ and a family $\pi_X\,:\,\mathcal{X}\,\rightarrow\,U$ such that
\begin{itemize}\renewcommand{\labelitemi}{$\bullet$}
\item The central fibre $X_0$ is a symplectic resolution of $S_0^{[2]}$,
\item For generic $t$ the point $X_t$ is a symplectic resolution of a double EPW-sextic with $k$ singular $K3$ surfaces.
\end{itemize}
 
\end{prop}

\subsection{Moduli spaces of stable objects}\label{ssec:stab_obj}
In this section we recall basic definitions and facts about moduli spaces of sheaves and Bridgeland stable objects on $K3$ surfaces. For a more detailed treatment of the latter the interested reader is referred to \cite{Bri08}.

Let $S$ be a projective $K3$ surface. Mukai defined a lattice structure on $H^*(S,\Z)$ by setting
\[(r_1,l_1,s_1).(r_2,l_2,s_2):=l_1\cdot l_2-r_1s_2-r_2s_1,\]
where $r_i\in H^0,$ $l_i\in H^2$ and $s_i\in H^4$. This lattice is referred to as the \em Mukai lattice \em and we call vectors $v\in H^*(S,\Z)$ \em Mukai vectors. \em The Mukai lattice is isometric to $\Lambda_{24}$.

Furthermore we may introduce a weight-two Hodge structure on $H^*(S,\Z)$ by defining the $(1,1)$-part to be
\[H^{1,1}(S)\oplus H^0(S)\oplus H^4(S).\]

For an object $\mc{F}\in D^b(S)$ we define the \em Mukai vector of $\mc{F}$ \em by
\[v(\mc{F}):=\ch(\mc{F})\sqrt{\td_S}=(\rk\mc{F},c_1(\mc{F}),\ch_2(\mc{F})+\rk\mc{F}).\]
It is of $(1,1)$-type and satisfies
\renewcommand{\labelitemi}{$\bullet$}
\begin{itemize}
\item  $r>0$ or
\item $r=0$ and $l\neq0$ effective or
\item $r=l=0$ and $s>0.$
\end{itemize}
\begin{defn}
A non-zero vector $v\in H^*(S,\Z)$ satisfying $v^2\geq 2$ and the conditions above is called a \em positive \em Mukai vector.
\end{defn}

With this definition we can easily deduce:

\begin{lem}\label{lem:posvec}
Let $v\in H^*(S,\Z)$ be non-zero and of $(1,1)$-type satisfying $v^2\geq 2$. Then either $v$ or its negative is a positive Mukai vector.
\end{lem}

Let us now review some results on the birational geometry of moduli spaces of bridgeland stable objects on a $K3$ surface. Let $S$ be a projective $K3$ surface and fix two classes $\beta, \omega\in \mathrm{NS}(S)_{\mathbb{R}}$ with $\omega$ ample. To this data Bridgeland associates a stability condition $\tau:=\tau_{\beta,\omega}$ on the derived category $D^b(S).$ The set of all such stability conditions $\tau_{\beta,\omega}$ is denoted by $\mathrm{Stab}(S).$ Next, we fix a primitive positive Mukai vector $v\in H^\ast(S,\mathbb{Z})$ and assume that $\tau$ is generic with respect to $v.$ The coarse moduli space $M_\tau(v)$ of $\tau$-stable objects of Mukai vector $v$ is a projective manifold of \kntiposp (\cite[Thm.\ 5.9]{BM12}) and we have an isometry of weight-two Hodge structures
\[H^2(M_\tau(v),\mathbb{Z})\stackrel{\sim}\rightarrow v^\perp\subset H^\ast(S,\mathbb{Z}).\]
Bayer and Macri studied the birational geometry of these moduli spaces; in particular, they introduced a chamber structure on $\mathrm{Stab}(S).$ We summarise their results:
\begin{thm}\cite[Thm.\ 1.1a) and Thm.\ 1.2]{BM13}\label{thm:bmstabsurj}
\begin{enumerate}
 \item If $\tau$ and $\tau'$ are generic stability conditions with respect to $v,$ then $M_\tau(v)$ and $M_{\tau'}(v)$ are birational.
 \item There is a surjective map
 \[l\colon\mathrm{Stab}(S)\rightarrow \mathrm{Mov}(M_\tau(v))\]
 mapping every chamber of $\mathrm{Stab}(S)$ onto a \kahl-type chamber such that for a generic $\tau'$ the moduli space $M_{\tau'}(v)$ is the birational model of $M_\tau(v)$ corresponding to the chamber containing $l(\tau').$
\end{enumerate}
\end{thm}

Note that for every positive Mukai vector at least one chamber in $\mathrm{Stab}(S)$ contains stability conditions $\tau_{\beta,\omega}$ whose stable objects are (up to a shift) stable sheaves in the sense of Gieseker.

Let $f\colon H^2(M_\tau(v),\Z)\rightarrow N$ be a marking and denote by $\mc{P}$ the period map (restricted to a connected component of the moduli space of marked manifolds). The above theorem implies, in particular, that every manifold in the fibre $\mc{P}^{-1}(\mc{P}(M_\tau(v),f)),$ is again a moduli space of stable objects on $S$ with the same Mukai vector $v.$

We even have the following stronger result:
\begin{cor}\label{cor:hodge_moduli_space}
Let $X$ and $X'$ be Hodge isometric manifolds of \kntipo. Then $X$ is a moduli space of stable objects on a $K3$ surface if and only if the same holds for $X'$.
\end{cor}
\begin{proof}
Let $X$ be a moduli space of stable objects on a $K3$. We study the corresponding fibre of the period map. The number of connected components (neglecting the additional components that we can reach by composing with $-\id$) of the moduli space of marked manifolds is in one-to-one correspondence to the number of choices for Mukai vectors $\{v_i\}_i$ such that for generic stability conditions $\tau_i$ the moduli spaces $M_{\tau_i}(v_i)$ are Hodge isometric. The corresponding embeddings of $H^2(M_{\tau_i}(v_i))$ into $\Lambda_{24}$ are pairwise not conjugate, therefore the moduli spaces are pairwise not birational. Thus, by the theorem above, all manifolds Hodge isometric to $X$ are moduli spaces.
\end{proof}

\begin{rmk}\label{rmk:mod_stab_obj_ab}
A very similar construction can be done in the case when $A$ is an abelian surface. Again, $M_\tau(v)$ is a projective manifold and the fibre $K_\tau(v)$ of the Albanese map $M_\tau(v)\rightarrow A\times \hat{A}$ is of Kummer $n$-type (\cite[Thm.\ 1.9]{Yos12}). Again, the second cohomology of $K_\tau(v)$ is Hodge isometric to $v^\perp$ and we have an analogous result as in \Ref{thm}{bmstabsurj}. The only important difference is the following: By \cite[Lem.\ 3]{Shi78} for every $2$-torus $A$ there is a Hodge isometry $g$ to its dual $\hat{A}$. Thus the moduli space of marked $2$-tori has four connected components (corresponding to $(A,f),$ $(A,f\circ -\id),$ $(\hat{A},f\circ g)$ and $(\hat{A},f\circ g\circ -\id)$, where $f$ is some marking of $A$). For every Mukai vector $v=(r,l,s)\in H^*(A,\Z)$ we define its dual as $\hat{v}:=(r,g_*l,s).$ We see immediately that $v$ is positive if and only if $\hat{v}$ is positive and the corresponding Albanese fibres $K_\tau(v)$ and $K_{\hat{\tau}}(\hat{v})$ are Hodge isometric. (Here $\hat{\tau}$ is the 'dual' stability condition on $\hat{A}$ defined in the obvious way.) Note that in general $K_\tau(v)$ and $K_{\hat{\tau}}(\hat{v})$ are not birational (\cite{Nam02}) but we, again, see that the moduli space of marked manifolds of Kummer $n$-type has (at least) four components. Summarising, we can state that the above corollary holds also for manifolds of Kummer $n$-type if $n+1$ is a prime power.
\end{rmk}

\subsection{Induced Automorphisms on Moduli Spaces}\label{ssec:ind_aut}
We want to study automorphisms on the moduli spaces described above. Thus we include a direct generalisation of Section 3 in \cite{OW13} to the case of moduli of Bridgeland stable objects. Let $S$ be a projective $K3$ surface, $v$ be a positive Mukai vector, $\tau=\tau_{\beta,\omega}$ a $v$-generic stability condition and $\varphi$ an automorphism of $S$.

\begin{prop}\label{prop:exist_induced}
Keep notation as above and assume that $v$, $\beta$ and $\omega$ are $\varphi$-invariant. Then $\varphi$ induces an automorphism $\widehat{\varphi}$ of the moduli space $M_\tau(v).$ If $\tau$ is not $\varphi$-invariant, we at least get a birational selfmap of $M_\tau$.
\end{prop}
\begin{proof}
This is the analogue of \cite[Prop.\ 3.1]{OW13}. We only need to check that the pullback along $\varphi$ induces an automorphism of the moduli functor. Indeed, from the definition of stability, it is clear that if an object $\mc{F}\in D^b(S)$ is $\tau$-stable, so is $\varphi^*\mc{F}$.
\end{proof}

\begin{rmk}
Note that we will typically consider a surface $S$ together with a non-symplectic automorphism $\varphi$ fixing $\NS(S)$. In this case all stability conditions are $\varphi$-invariant.
\end{rmk}

In order to be able to study the induced action of $\widehat{\varphi}$ on the second cohomology we need the following:

\begin{lem}\label{lem:equivariant_induced_iso}
The isomorphism of Hodge structures
\[H^2(M_\tau(v),\Z)\rightarrow v^\perp\]
is $(\varphi,\widehat{\varphi})$-equivariant. In particular, we have
\[ H^2(M_\tau(v),\Z)^{\widehat{\varphi}}\cong (v^\perp)^\varphi.\]
\end{lem}
\begin{proof}
Using Definition 5.4 of the above isomorphism in \cite{BM12}, we only need to check that for any curve $C\subset M_\tau(v)$ we have
\[\Phi_{\mc{E}}(\mc{O}_{\widehat{\varphi}^*C})\simeq \varphi^*\Phi_{\mc{E}}(\mc{O}_C),\]
where $\Phi_{\mc{E}}$ is the Fourier$-$Mukai transform associated with a quasi universal family $\mc{E}$ on $M_\tau(v)\times S$. But this follows easily from the fact that, by the definition of $\hat{\varphi}$, $\mc{E}$ is $(\widehat{\varphi}\times\varphi)$-invariant.
\end{proof}

As usual, the results can be translated in the appropriate sense to the case of abelian surfaces $A$ and their moduli. If we consider automorphisms $\varphi$ of $A$ which preserve the origin (\ie homomorphisms), a straightforward calculation shows that the induced automorphism respects the fibre $K_\tau(v)$ over $(0,0)$ of the Albanese map $M_\tau(v)\rightarrow A\times\hat{A}$ and we obtain an induced action on the fibre. The other important class of automorphisms on $A$ are translations. By \cite{BNS11} we know that we have an action of the group of $n+1$-torsion points on $K_\tau(v)$, where $\dim K_\tau(v)=2n$.

\section{Periods of Moduli Spaces}\label{sec:periods_mod_spac}


This section is devoted to answering the following question: How can we determine if a given manifold of \kntiposp is, in fact, a moduli space of stable objects on some K3 surface? We state a necessary and sufficient criterion entirely in terms of lattice theory.

\begin{defn}
Let $X$ be a projective manifold of \kntiposp ($n\geq2$) and let $i$ be a primitive embedding of $H^2(X,\Z)$ into the Mukai lattice $\Lambda_{24}$. Endow the latter with  the (unique) Hodge structure making $i$ an embedding of Hodge structures such that the complement of the image of $i$ is of type $(1,1)$. We call $X$ a \em numerical moduli space \em if $\Lambda_{24}^{1,1}$ contains the hyperbolic plane $U$ as a direct summand.
\end{defn}

\begin{rmk}
Note that, in general, there is more than one conjugacy class of embeddings $i$ as above. The condition to be a numerical moduli space, however, is independent of the choice of an embedding.
\end{rmk}

\begin{prop}\label{prop:k3n_moduli}
Let $X$ be a manifold of \kntiposp which is a numerical moduli space.  Then there exists a projective $K3$ surface $S$ such that $X$ is, in fact, a moduli space of stable objects of $D^b(S)$ for some stability condition $\tau\in \mathrm{Stab}(S)$.  
\end{prop}
\begin{proof}
The complement of the hyperbolic plane in $\Lambda_{24}$ is a lattice of signature $(3,19)$ and it inherits in a natural way a weight two Hodge structure. By the surjectivity of the period map for $K3$ surfaces there exists a surface $S$ such that $\Lambda_{24}^{1,1}\cong U\oplus \NS(S)$. From the hyperbolicity of $\NS(S)$ it follows that $S$ is algebraic and furthermore we may choose an isometry $H^*(S,\Z)\cong \Lambda_{24}$ respecting the given embedding of $H^2(S,\Z)$. The orthogonal complement of $H^2(X,\Z)$ inside the Mukai lattice is a rank one lattice. The square of a generator has length $2n-2$. By \Ref{lem}{posvec} we obtain a positive Mukai vector $v$. Let $\tau$ be a $v$-generic stability condition on $S$ and set $M:=M_\tau(v)$. It is a manifold of \kntiposp which is, by construction, Hodge isometric to $X$ such that the embeddings of $H^2(X,\Z)$ and $H^2(M_\tau(v))$ into the Mukai lattice are equivalent. Thus by \cite[Cor.\ 9.9 (1)]{Mar11} $X$ and $M_\tau(v)$ are birational. We may conclude using the results of \cite{BM13}, which we summarised in \Ref{ssec}{stab_obj}.
\end{proof}

We remark that an independent proof of the above result has been given by Addington \cite[Proposition 4]{Add14} while this work was being prepared.\\
Copying the above definition to the case of Kummer $n$-type manifolds (and replacing $\Lambda_{24}$ by $\Lambda_8$), we have the following result:

\begin{prop}\label{prop:kummer_moduli}
Let $X$ be a manifold of Kummer $n$-type with $n+1$ a prime power. Assume that $X$ is a numerical moduli space. Then there exists an abelian surface $A$ such that $X$ is, in fact, the fibre of the Albanese map of a moduli space of stable objects of $D^b(A)$ for some stability condition $\tau\in \mathrm{Stab}(A)$.  
\end{prop}
\begin{proof}
It is the lacking of a precise description of the monodromy group of manifolds of Kummer $n$-type for general $n$ which prevents us from proving the above theorem in general. However, the partial description given in \Ref{prop}{mono_kum} is enough to prove the theorem if $n+1$ is a prime power: there are four connected components of the moduli space of marked manifolds of Kummer $n$-type. Now, as in the case of \kntiposp manifolds, we find an abelian surface and a Mukai vector such that the associated (Albanese fibre of the) moduli space is Hodge isometric to $X$. Together with its birational models it covers all points in the fibre $\mc{P}^{-1}(\mc{P}(X,f)),$ which belong to two out of the four connected components. (Here $f$ is some marking of $X$. Note that we may always compose a marking with $-\id$ to end up in a second connected component.) The points in the fibre, which belong to the other two connected components, correspond to (Albanese fibres of) moduli spaces of stable objects on the dual abelian surface (cf. \Ref{rmk}{mod_stab_obj_ab}). 
\end{proof}

\begin{conj}
The above statement holds, in fact, for all $n$.
\end{conj}



\section{Induced Automorphism Groups}\label{sec:induced_auto_group}
Next, let us consider what happens in the presence of automorphisms of the manifolds. We start by a definition.

\begin{defn}
Let $X$ be a manifold of \kntiposp and let $G\subset \aut(X)$. We say that $G$ is an \em induced group of automorphisms \em if there exists a projective $K3$ surface $S$ with $G\subset \aut(S)$, a $G$-invariant Mukai vector $v\in H^*(S,\Z)^G$ and a $v$-generic stability condition $\tau$ such that $X\cong M_\tau(v)$ and the induced (maybe apriori birational) action on $M_\tau(v)$ (cf.\ \Ref{prop}{exist_induced}) coincides with the given action of $G$ on $X$.
\end{defn}
The same definition can be given for manifolds of Kummer-type.
\begin{defn}\label{defn:k3n_induced}
Let $X$ be a manifold of \kntiposp and let $G\subset \aut(X)$. Let $i$ be a primitive embedding of $H^2(X,\mathbb{Z})$ inside the Mukai lattice $\Lambda_{24}$. Then the group $G$ is called \emph{numerically induced} if the following hold:
\begin{itemize}
\item  The group $G$ acts trivially on $A_X$, so that its action can be extended to the Mukai lattice with $S_G(\Lambda_{24})\cong S_G(X)$ as in \Ref{lem}{aut_extend}. 
\item The $(1,1)$-part of the lattice $T_G(\Lambda_{24})$ contains the hyperbolic plane $U$ as a direct summand. 
\end{itemize} 
\end{defn}

\begin{rmk}
The two definitions above can be phrased in a very similar fashion for manifolds of Kummer- and of $Og_{10}$- and $Og_6$-type (cf.\ \Ref{sec}{kieran_case} for the latter two). Furthermore note that it is easy to see that an induced group of automorphisms is numerically induced.
\end{rmk}


\begin{thm}\label{thm:k3n_induced}
Let $X$ be a manifold of \kntiposp and let $G\subset \aut(X)$ be a numerically induced group of automorphisms. Then there exists a projective $K3$ surface $S$ with $G\subset \aut(S)$, a $G$-invariant Mukai vector $v$ and a $v$-generic stability condition $\tau$ such that $X\cong M_\tau(v)$ and $G$ is induced.  
\begin{proof}
First of all, let us consider the case $G$ symplectic. Then we have $S_G(X)\subset \NS(X)$ and $S_G(X)$ contains no elements of square $-2$. Since $G$ is numerically induced, let us write $T_G(\Lambda_{24})=U\oplus T$. We then have that $S_G(X)$ embeds in the $K3$ lattice and its orthogonal is $T$, where the action of $G$ is trivial. Let us give this $K3$ lattice the induced Hodge structure from $\Lambda_{24}$ and let $S$ be the corresponding $K3$ surface. By \Ref{prop}{k3n_moduli}, $X$ is a moduli space of stable objects on $S$. Now $G$ is a group of Hodge isometries on $S$ which contains no elements generated by reflections on elements of square $-2$, therefore $G\subset \aut(S)$ and its induced action on $H^2(X)$ is the one we started with. Since the representation of automorphisms on $H^2$ is faithful, we are done.\\
Now let us suppose that no non-trivial element of $G$ preserves the symplectic form. This implies $T_G(X)\subset \NS(X)$. Without loss of generality we can suppose $T_G(X)=\NS(X)$. As before we let $T_G(\Lambda_{24})=U \oplus T$ and we let $S$ be the $K3$ surface associated to the Hodge structure defined on $U^\perp$ inside $\Lambda_{24}$. Again by \Ref{prop}{k3n_moduli} $X$ is a moduli space of stable objects on $S$ and $G$ is a group of Hodge isometries of $S$ preserving $T=\NS(S)$, therefore $G\subset \aut(S)$ and its action on $X$ coincides with the induced one.\\
Finally, if $G_s$ is the symplectic part of $G$, we obtain a $K3$ surface $S$ as in the first step with $G_s\subset \aut(S)$. We can extend also the action of $G$ on $S$ by applying the second step to the group $\overline{G}=G/G_s$.    
\end{proof}
\end{thm}

\begin{thm}
Let $X$ be a manifold of Kummer-type with $n+1$ a prime power. Let $G\subset \aut(X)$ be a numerically induced group of automorphism. Then there exists an abelian surface $A$ with $G\subset \aut(A)$, a $G$-invariant Mukai vector $v$ and a stability condition $\tau$ such that $X$ is the Albanese fibre of $M_\tau(v)$ and $G$ is induced.  
\end{thm}
\begin{proof}
The proof is literally the same as the one of \Ref{thm}{k3n_induced}. We remark here that all these manifolds have non-trivial automorphisms acting trivially on the second cohomology. However, these automorphisms deform smoothly on all manifolds of Kummer-type.
\end{proof}

In the special case of finite groups of automorphisms preserving the symplectic form, the notions of induced and natural morphisms coincide, as the following more general result shows:

\begin{thm}\label{thm:sympl_def}
Let $\mathcal{M}_{N}$ be the moduli space of marked \issp manifolds with $H^2\cong N$. Let $G$ be a finite group of symplectic automorphisms acting on some $(X,f)\in\mathcal{M}_N$ and let $G$ contain all morphisms acting trivially on $H^2$. Then the number of connected components of the moduli space of manifolds having the same $G$-action as $X$ on $H^2$ is bounded by the number of connected components of $\mathcal{M}_N$.
\begin{proof}
Let $Y$ be another manifold such that $G$ acts on its second cohomology as it acts on $X$, \ie there exists an isometry $f\colon H^2(X,\mathbb{Z})\rightarrow H^2(Y,\mathbb{Z})$ which restricts to an isometry of $T_G(X)$ and $T_G(Y)$. In this way we can define a moduli space of marked manifolds having $S_G(X)$ primitively embedded inside the Picard lattice. To prove our claim, it is enough to prove that $X$ can be deformed into $Y$ if $f$ is a parallel transport operator, so we assume it is from now on.

Let $\alpha$ be a $G$-invariant \kahl class on $X$ and let $TW_\alpha(X)$ be the twistor family associated to it. The action of $G$ extends fibrewise to $TW_\alpha(X)$, since a \kahl class is always preserved. If we let $\Omega_G$ be the period domain associated to the lattice $T_G(X)$, we obtain that it has signature $(3,\mathrm{rank}-3)$ and is therefore connected by twistor lines (see \cite[Proposition 3.7]{Huy10}. These lines lift to twistor families given by $G$-invariant \kahl classes on the preimage of the period map, as is done in \cite{Mon13} in the \kntiposp case. In this way we can deform $(X,G)$ to a pair $(Z,G)$ with the same period of $Y$. The final deformation step can be done using local $G$-deformations. The very general point of such a deformation has $\NS\cong S_G(X)$ and the positive cone coincides with the \kahl cone, therefore the local $G$-deformations of $Z$ and $Y$ intersect and the action of $G$ on $Y$ coincides on $H^2$ with the action deformed from $X$. Now, since $G$ contains all morphisms acting trivially on $H^2$, the action we deformed to $Y$ coincides with the initial one.
\end{proof}
\end{thm} 


\begin{ex}
In this small example we want to illustrate that the first condition in \Ref{defn}{k3n_induced} is necessary. Thus we are looking for an automorphism of a manifold of \kntipo, such that the induced action on the discriminant of $H^2$ is non-trivial.

Let us first recall an example of a birational self-map that has been studied first by Beauville \cite{Beau83b}. Let $S$ be a K3 of degree $6$, that is, S is the intersection of a quadric and a cubic in $\PP^4$. A general element in $S^{[3]}$ defines a $2$-plane in $\PP^4$ and its intersection with $S$ is a scheme of length $6$. Thus, taking the residual length $3$ subscheme defines a birational involution $S^{[3]} \dashedrightarrow S^{[3]}$. Following Debarre (\cite[Th\'{e}or\`{e}me 4.1]{Deb84}), the action on $H^2$ is given by the reflection at the element $\alpha_0:=H-\delta$. In particular, it sends $\delta$ to $4H-5\delta$, where $H$ is the class of a hyperplane section. Thus the generator $[\delta/4]$ of the discriminant group is mapped to its negative.

A detailed analysis of the known results on the precise structure of the movable and the ample cone shows the following: The closure of the  movable cone of $S^{[3]}$ is spanned by $H$ and $5H-6\delta$. Thus it contains the class $\alpha_0$ (which is invariant under the involution). However, $\alpha_0$ is orthogonal to $2H-3\delta$, which is a class of square $-12$ and divisibility $2$, hence a wall of the movable cone. This implies that this involution is not regular on any manifold $X$ birational to $S^{[3]}$.\\ As it is often the case, we can obtain a regular involution by paying the price of dealing with a singular manifold (which is just the contraction of the $\mathbb{P}^3$ giving the indeterminacy locus). This manifold can be seen as a moduli space of stable object on the $K3$ surface $S$ with a bad choice of a polarisation, see \cite{Mat14} for a closely related example. \\
Nevertheless, this example can be seen as a degeneration of a twenty dimensional family of manifolds with an involution acting as $-1$ on the discriminant group. Indeed, if we take local deformations of $S^{[3]}$ where $H-\delta$ stays algebraic, the very general element of this family has ample cone which coincides with the positive cone and the birational involution deforms to a regular involution on it.


\end{ex}

\section{O'Grady's Examples}\label{sec:kieran_case}
In this section we introduce a notion of induced automorphisms and give a numerical criterion for manifolds of $Og_{10}$ or $Og_{6}$-type. Part of our results are based on \cite{MW14}, where the kernel of the cohomological representation $\nu\,:\, Aut(X) \rightarrow O(H^2(X))$ is determined for $X$ a manifold as above.

We fix a projective $K3$ surface $S$ and consider a primitive Mukai vector $w\in H^*(S,\Z)$ of square $2$. Set $v=2w$. The moduli space $M(2w)$ (we tacitly assume the choice of some $w$-generic polarisation) is of dimension $10$ and its singular locus $\Sigma:=M(2w)_{sing}$ consists exactly of the points corresponding to $S$-equivalence classes of sheaves of the form $\mc{F}_1\oplus \mc{F}_2$, where $\mc{F}_1$ and $\mc{F}_2$ are stable sheaves in $M(w)$. The blow up $\widetilde{M}(2w)$ of $M(2w)$ along $\Sigma$ with its reduced scheme structure yields a symplectic resolution (cf.\ \cite{LS06}). Perego and Rapagnetta showed that $M(2w)$ is a $2$-factorial symplectic manifold admitting an appropriate analogue of a Beauville$-$Bogomolov form and a pure weight-two Hodge structure on $H^2(M(2w),\Z)$. 
 Furthermore they computed the Hodge and lattice structure on $H^2(\widetilde{M}(2w),\Z)$. They introduced the following notation: Set
\[\Gamma_v:= \{(\alpha,k\sigma/2)\in (v^\perp)^\vee\oplus_\perp \Z\sigma/2\mid k\in 2\Z \Leftrightarrow \alpha\in v^\perp\}.\]
By declaring $\sigma$ to be of type $(1,1)$ and setting $\sigma^2:=-6$ we obtain a pure weight-two Hodge structure and lattice structure on $\Gamma_v$ induced by the corresponding structures on $H^*(S,\Z)$. Note that, by definition, we have $\sigma^\perp (\subset\Gamma_v)\cong v^\perp (\subset H^*(S,\Z)).$

\begin{thm}
There is a Hodge isometry of pure weight-two Hodge structures
\begin{equation}\label{eqhodgeisom}
\Gamma_v  \cong H^2(\widetilde{M}(2w),\Z),
\end{equation}
where the lattice structure on the right hand side is the Beauville$-$Bogomolov form.
\end{thm}
\begin{proof}
\cite[Theorem 3.4]{PR13}
\end{proof}

Now, assume we have an automorphism of the surface $S$. It naturally induces an isometry of $H^*(S,\Z)$. If $\varphi$ fixes the Mukai vector $w$, we obtain an isometry of $\Gamma_v$ by demanding $\sigma$ to be fixed.

\begin{defn}
Let $\widetilde{M}(2w)$ be a desingularised moduli space of $Og_{10}$-type and let $G\subset \aut(\widetilde{M}(2w))$. We call $G$ \em numerically induced \em if the following hold:
\begin{itemize}
\item The induced action on $H^2(\widetilde{M}(2w),\Z)$ fixes $\sigma$ (cf.\ eq.\ (\ref{eqhodgeisom} )).
\item The induced action on $\sigma^\perp$ is numerically induced in the sense of definition \Ref{defn}{k3n_induced}.
\end{itemize}
\end{defn}

\begin{prop}\label{prop:ogrady_induced}
Let $\varphi$ be an automorphism of the surface $S$ fixing the Mukai vector $w$ (and the chosen polarisation). It induces an automorphism $\tilde{\varphi}$ on the desingularised moduli space $\widetilde{M}(2w)$ such that the isometry (\ref{eqhodgeisom}) is $\varphi-\tilde{\varphi}$-equivariant. In particular, the action of $\tilde{\varphi}$ is numerically induced.
\end{prop}
\begin{proof}
By \Ref{prop}{exist_induced} we obtain an induced automorphism $\hat{\varphi}$ on the singular space $M(2w)$ and by \Ref{lem}{equivariant_induced_iso} the Hodge isometry $v^\perp\cong H^2(M(2w),\Z)$ is $\varphi-\hat{\varphi}$-equivariant. Furthermore the singular locus $\Sigma$ is certainly $\hat{\varphi}$-invariant and we have a well defined induced action on the normal bundle $\mc{N}:=\mc{N}_{\Sigma|M(2w)}$. Indeed, the fibre $\mc{N}_{\mc{F}_1,\mc{F}_2}$ of $\mc{N}$ over $\mc{F}_1\oplus \mc{F}_2$ is isomorphic to $\Ext^1(\mc{F}_1,\mc{F}_2)\oplus \Ext^1(\mc{F}_2,\mc{F}_1)$ and the map $\mc{N}_{\mc{F}_1,\mc{F}_2}\rightarrow \mc{N}_{\varphi^*\mc{F}_1,\varphi^*\mc{F}_2}$ is the obvious pullback map. Thus we get an induced action of $\varphi$ on the blow up $\widetilde{M}(2w)$. We immediately deduce that the  class $\sigma$ corresponding to the exceptional divisor of the blow up is fixed by the induced action.
\end{proof}

\begin{defn}
Let $\widetilde{M}(2w)$ be the symplectic resolution of the singular moduli space $M(2w)$ for a primitive Mukai vector $w$ of square $2$ on a projective $K3$ surface $S$. Let $G\subset \aut(\widetilde{M}(2w)$. We say that $G$ is an \em induced group of automorphisms \em if $G\subset \aut(S)$ with $w\in H^*(S,\Z)^G$ and the induced action on $\widetilde{M}(2w)$ coincides with the given one.
\end{defn}

\begin{prop}
Let $w$ be a primitive Mukai vector of square $2$ and let $\widetilde{M}(2w)$ be the symplectic resolution of singularities of $M(2w)$ for a projective $K3$ surface $S$. Let $G\subset \aut(\widetilde{M}(2w))$ be a finite subgroup. Then $G\subset \aut(S)$ and $G$ is induced if and only if it is numerically induced.
\end{prop}
\begin{proof}
The `only if'-part is the content of \Ref{prop}{ogrady_induced}. For the other direction let $\psi$ be an automorphism of $\widetilde{M}(2w)$ which is numerically induced. We thus obtain an induced isometry of $\sigma^\perp$ which is induced in the sense of \Ref{defn}{k3n_induced}. Thus we obtain an isomorphism $\varphi$ on the underlying surface $S$. Now, $\psi$ and the induced automorphism $\tilde{\varphi}$ on $\widetilde{M}(2w)$ have the same action on the second integral cohomology. By \cite{MW14}, the cohomological representation is injective, hence we are done.
\end{proof}


\begin{rmk}\label{rmk:lattices_kum_ogrady}
Statements similar to \Ref{prop}{k3n_moduli} and \Ref{prop}{kummer_moduli} can be expected to hold true for O'Grady-type manifolds: Let $X$ be deformation equivalent to $Og_{10}$. 
Embed $H^2(X,\Z)$ into $\Lambda_{24}\oplus U$ 
 and endow the latter with the induced Hodge structure. Suppose the $(1,1)$-part of $\Lambda_{24}\oplus U$ 
  contains $U^{\oplus 2}$ 
   as a direct summand. Then we expect that $X$ is given as the desingularisation of 
    a moduli space of stable objects on a $K3$ 
     surface.
\end{rmk}

Copying the appropriate definitions from the $Og_{10}$-case, we can conclude in a similar way in the abelian surface case:

\begin{cor}
Let $w$ be a primitive Mukai vector of square $2$ on a projective abelian surface $A$ and let $\widetilde{K}(2w)$ be the desingularised Albanese fibre of $M(2w)$. Let $G\subset \aut(\widetilde{K}(2w))$ be a finite subgroup. Then $G\subset\aut(A)\times A^\vee[2]$ and $G$ is induced if and only if it is numerically induced.
\end{cor}

\begin{rmk}
We remark that in the $Og_{6}$-type case the kernel of the cohomological representation $\nu$ is $A[2]\times A^\vee[2]$, as proven in \cite{MW14}, and this explains the presence of $A^\vee[2]$ in the above corollary.
\end{rmk}

\section{Applications}\label{sec:appl}

\subsection{Non-symplectic automorphisms of prime order on manifolds of \ktipo}
In \cite{BCS14} non-symplectic automorphisms of prime order (different from $5$) on manifolds of \ktiposp were analysed and a full lattice-theoretic picture was given. If the order $p$ of the automorphism is equal to $7,$ $11,$ $13,$ $17$ or $19,$ they have shown that all\footnote{Note that there is one interesting case for $p=13$ where it is not known whether a geometric realisation exists.}lattices can be realised by natural automorphisms.  Fixing one such natural automorphism, by a similar reasoning as in the proof of Corollary 2.11 in \cite{OW13} we see that any other family of automorphisms with an equivalent action on $H^2$ is obtained as a (global) flop. The flopped family consists of moduli spaces of stable objects by \cite{BM13} and the automorphisms are induced. 

The same holds true for most of the cases if $p=3$. There are two remaining cases corresponding to the lattices $T\cong \langle 6\rangle$ and $T\cong \langle 6\rangle\oplus E_6^\vee(-3)$. Both lattices are realised as invariant lattices of automorphisms on the Fano variety of lines on a cubic fourfold (cf.\ \cite[Sec.\ 6.2]{BCS14}).

\begin{lem}
Both lattices above do not correspond to induced automorphism.
\end{lem}
\begin{proof}
We compute the invariant lattice of the induced isometry $g$ in $\Lambda_{24}$ by hand. In both cases we are looking for a $3$-elementary lattice of rank equal to $\rk T+1$ and of signature $(2,\rk-2)$. For $T\cong \langle 6\rangle$ we get $T_g(\Lambda_{24})\cong A_2$ and for $T\cong \langle 6\rangle\oplus E_6^\vee(-3)$ we have $T_g(\Lambda_{24})\cong A_2\oplus E_6^\vee(-3)$. Both lattices do not contain the hyperbolic plane, whence the lemma.
\end{proof}

Let us proceed to the study of automorphisms of order $2$. The action of non-symplectic involutions on the second cohomology of manifolds of \ktiposp can be fully classified with four invariants: the rank $r$ of the invariant lattice, the length $a$ of its discriminant group, the discrete invariant $\delta$ used for $2$-elementary lattices and a further discrete invariant which tells us whether the action is given as the action of a natural automorphism. Since the latter case gives always induced morphisms, the most interesting case is given by involutions which are not natural cf.\ \cite[Fig.\ 1]{BCS14}). We divide the big number of cases as follows:
\begin{itemize}
\item $a<r-3$: In this case $T$ contains a copy of $U$, thus the corresponding automorphisms are all induced. Nevertheless some interesting geometry can be observed, cf.\ \Ref{ex}{example_E_8}.
\item $a=r-2$: The cases $a=0$ and $a>0,$ $\delta_T=1$ correspond to the lattices $U\oplus \langle-2\rangle^{\oplus a}$, so the automorphisms are all induced.
\item $a=r,$ $\delta_T=1$: These examples can all be we realised by degenerations of double EPW-sextics, see \Ref{ssec}{ssec_EPW}.
\item The remaining cases are the lattices $U(2)$ and $U(2)\oplus E_8(2)$ for $a=r,$ $\delta_T=0$ and $U(2)\oplus D_4$ and $U\oplus E_8(2)$ for $0>a=r-2,$ $\delta_T=0$, the latter of which obviously corresponds to an induced automorphism.
\end{itemize}

\begin{rmk}
In \cite{OW13} the authors proved that for the natural involution having invariant lattice $\langle 2\rangle\oplus\langle-2\rangle$, the automorphism carries over to the flop. Note that flops are the only birational maps between manifolds of \ktiposp and they are given by reflecting at a wall orthogonal to a $-10$ class of divisor $2$. Since all the examples above, by construction, do not contain any algebraic classes of divisor $2$, no flops of the families exist. 
\end{rmk}

\begin{cor}
With the exception of three cases, there exists a geometric realisation of all families of prime order (different from $5$ and $23$) non-symplectic automorphisms on manifolds of \ktipo.
\end{cor}

\begin{rmk}\label{rmk:assignment}
It can be an interesting task to actually construct a $K3$ surface such that a given family of induced automorphisms as above can be realised by a moduli space of sheaves on it.

In general, we are given a manifold $X$ together with an induced involution with $2$-elementary lattice $T(X)$ with invariants $(r,a,\delta)$. Then it is easy to see that the invariant lattice $T(\Lambda_{24})$ has invariants $(r+1,a+1,1)$ and has signature $(2,r-1)$. Splitting off a copy of $U,$ we deduce that the invariant lattice $T(S)$ of the corresponding $K3$ surface $S$ must have invariants $(r-1,a+1,1)$. In this way we obtain an assignment from the set of lattices of the induced examples to the famous diagram of Nikulin \cite{Nik79} listing the invariant lattices of non-symplectic involutions on $K3$s. Note, that $\delta_{T(S)}$ is always equal to one, regardless of the value $\delta_{T(X)}$. Thus the above mentioned assignment is not injective. It says that on the same $K3$ surface we may choose two different Mukai vectors inducing different families in the classification.
\end{rmk}

\begin{ex}
Let us look a little closer at the case when $T_G(X)$ itself already contains a copy of the hyperbolic plane: $T_G(X)\cong U\oplus N,$ for some negative definite lattice $N$. Assume that we have $T_G(\Lambda_{24})\cong \langle2\rangle\oplus T_G(X)$ and $T_G(S)\cong \langle2\rangle \oplus N$. Thus $S$ is a (degenerate) double sextic $\pi\colon S\rightarrow \PP^2$ and the Mukai vector is orthogonal to the copy of $U$, e.g.\ $v=(0,H,0),$ where $H=\pi^*\mc{O}(1)$.
\end{ex}

\begin{ex}\label{ex:example_E_8}
Here we want to study the particularly nice cases when the invariants of $T_G(X)$ are $(10,8,\delta),$ $\delta=0,1$. This is one of the cases where the assignment mentioned in \Ref{rmk}{assignment} is not injective. It leads to a $K3$ with invariant lattice $T_G(S)\cong\langle2\rangle \oplus E_8(2)$. There is a very nice description of these $K3$ surfaces in \cite[Exa. 3.2]{vGS07}. These surfaces are double sextics (denote the covering involution by $i$) that additionally admit a symplectic involution $\iota$ commuting with $i$. The involution on $S$ that we have to consider is then the composition $\iota\circ i$. The two different Mukai vectors that we need to choose in order to construct the two families of induced automorphisms are as follows. First choose $v_1=(0,H,0)$ as above. Next, let $f\colon S\rightarrow S':=S/\iota\circ i$ denote the quotient. It is a degree $1$ del Pezzo surface. Denote by $O_{S'}(1)$ the pullback of the hyperplane section along the blow-ups. Then we choose $v_2:=(0,f^*O_{S'}(1),0).$ The map $f$, of course, exhibits the surface $S$ as a degenerate double sextic, where the branch curve has eight double points. These two different geometric points of view correspond to the lattice theoretic isometry $T_G(S)\cong \langle2\rangle\oplus E_8(-2)\cong \langle2\rangle\oplus \langle-2\rangle^{\oplus 8}$.
\end{ex}

\subsection{Non-induced involutions on double EPW-sextics}\label{ssec:ssec_EPW}
There are 14 families of non-symplectic involutions which are not induced. We will now contruct 11 of them as degenerations of double EPW-sextics and their canonical involution.
The construction is taken from \cite{Fer12} and we sketched it in \Ref{prop}{fer}. We take the $19-k$ dimensional family $\mathcal{B}_k$ of quartic surfaces in $\mathbb{P}^3$ having $k\leq 10$ ordinary double points in general position, and we consider the associated family of smooth $K3$ surfaces obtained by blowing up the singular points. The generic such surface has Picard lattice isometric to $(4)\oplus (-2)^{\oplus k}$. Then we take the family of their Hilbert schemes and we endow them with the involution sending a set of two points on a $K3$ surface $S$ to the other two points in the intersection $S\cap l$, where $l$ denotes the line the two points of $S$ span in $\PP^3$.  Applying \Ref{prop}{fer}, we extend the family $\mathcal{B}_k$ to a $20-k$ dimensional family $\mathcal{B}'_k$ whose generic member is the blowup of a double EPW-sextic in $k$ singular $K3$ surfaces of degree $2$, therefore it has a non-symplectic involution. Such involution fixes the class of the semiample divisor of these double sextics and fixes also all divisors corresponding to the resolutions of the singular $K3$s, which are linearly independent. Therefore the invariant lattice of the involution has rank $k+1$ and $\delta=1$. Moreover each singular $K3$ gives a divisor of square $-2$, as in \cite[Claim 3.8]{OGr12}, therefore we also have $a=r$.

\begin{rmk}
For $a=2$ the invariant lattice is isomorphic to $\langle2\rangle\oplus\langle-2\rangle$, where the generator of square $-2$ has divisibility $1$. Thus this family corresponds to one of the missing cases in the analysis of $19$-dimensional families of non-symplectic involutions in \cite{OW13}.
\end{rmk}

\begin{ex}
Let us consider the special case $k=8$. Here we have a family of double EPW sextics whose very general element has Picard lattice $(2)\oplus (-2)^8$. But this lattice is also isometric to $(2)\oplus E_8(-2)$, which is the lattice corresponding to a general double EPW sextic with an additional symplectic involution commuting with the covering involution, as constructed in \cite{Cam12}. Therefore we can see these double sextics as ordinary smooth double sextics with an appropriate change of polarization, much in analogy with the case of $K3$s of degree $2$ with a symplectic involution.
\end{ex}

\subsection{Lagrangian fibrations with a section}
Let $X$ be an \issp manifold of dimension $2n$ and let $\varphi\colon X\rightarrow B$ be a lagrangian fibration over a smooth space $B$. In \cite{GL13} the authors prove that $B=\mathbb{P}^n$. Let now $s\colon\mathbb{P}^n \rightarrow\,X$ be a section. If $X$ is a $K3$ surface, then the classes of the section and of the generic fibre generate a lattice isomorphic to $U$ inside the Picard group of $X$.\\
We can consider families of \issp manifolds with a lagrangian fibration with a section as a family of manifolds with a (maybe birational) non-symplectic involution $\alpha$, given by changing the sign on all smooth fibres. 
\begin{oss}
The existence of a section generically implies that the base of the fibration is smooth, without any assumption on it. Indeed, the section is one component of the fixed locus of the involution associated with the sign change, which is smooth for regular maps. This involution preserves the Picard lattice (hence the ample cone) under genericity assumptions and is therefore regular. With the additional hypothesis that the base of the fibration is normal, this was already proven in \cite[Theorem 7.2]{CMSB02} 
\end{oss}
An interesting remark is that having such a fibration imposes conditions on the Picard group also if the dimension of $X$ increases. Indeed, the presence of a lagrangian fibration implies that there exists a divisor $D\subset X$ equal to the pullback of the hyperplane section of $\mathbb{P}^n$. Since $D^n$ is the class of a fibre, we have $D^2=0$ with respect to the Beauville Bogomolov form. One more class in the Picard group is given by the section $s$: the class of a line $l$ inside $s(\mathbb{P}^n)$ is an extremal ray of the Mori cone and the intersection $D\cdot l$ is equal to $1$. Let $H_2(X,\mathbb{Z})\hookrightarrow H^2(X,\mathbb{Q})$ be the inclusion induced by the identification $H_2(X,\mathbb{Z})=H^2(X,\mathbb{Z})^\vee$ as lattices and let $T$ be the least positive multiple of $l$ which is a divisor, \ie $T/\div(T)=l$. Then we have $(T,D)=\div(T)$ and the lattice spanned by them is the invariant lattice of the non-symplectic involution $\alpha$.

\begin{prop}\label{prop:lagr_lattice}
Keep notation as above and assume that $X$ is a manifold of \kntipo. Then, if $n\leq 3$, the class of $T$ has square $-2(n+3)$ and $\div(T)=2$. For any $n$, the generic such $X$ has Picard group isometric to $U(2)$ if $n$ is odd and to $(2)\oplus (-2)$ otherwise.
\begin{proof}
As usual we embed $H^2(X,\mathbb{Z})$ inside $\Lambda_{24}$ and we extend the action of the isometry $\alpha$.
First of all we claim that $\div(T)\neq 1$: indeed the extremal ray $l$ corresponds to a birational transformation which is not divisorial. In \cite{Mon13b}, numerical orbits of extremal rays are classified and $\div(T)=1$ corresponds only to a divisorial contraction, as stated in \cite[Remark 2.5]{Mon13b} (see also \cite[Section 9]{Mar11}).
We must distinguish two cases: in the first one the involution $\alpha$ acts trivially on $A_X$ and in the second it acts as an involution.
\begin{itemize}
\item Let $\alpha$ act trivially on $A_X$. This implies that $\alpha$ extends trivially on $\langle v\rangle:=H^2(X,\mathbb{Z})^\perp$ inside $\Lambda_{24}$. By \Ref{oss}{G_tors}, $T_\alpha(\Lambda_{24})$ is $2$-elementary and is an overlattice of $\langle v, T, D\rangle$, where the latter has discriminant $-2(n-1)(\div (T))^2$ and its discriminant group is generated by $[v/(2n-2)],[D/\div(T)]$ and $[T/\div(T)]$. The lattice $T_\alpha(\Lambda_{24})$ contains $w:=(v+T)/\div(T)$ and the lattice $\langle w,D,T \rangle$ has discriminant $2(n-1)$. No further overlattice is contained in $\Lambda_{24}$. As in \Ref{oss}{overlattice}, the discriminant group contains $[(v-D)/div(T)]$, which must have order 1 or 2. Therefore $n=\div(T)=2$ and $T^2=-10$, which agrees with our claim.
\item Let $\alpha$ act as an involution on $A_X$. This implies that $\alpha$ acts as an involution on $H^2(X,\mathbb{Z})^\perp$ inside $\Lambda_{24}$, therefore $T_\alpha(X)=T_\alpha(\Lambda_{24})$ and it is $2$-elementary. This implies $\div(T)=2$.  Therefore $T^2\equiv-2(n+3)\,mod\,4$ and $\div(T)=2$, as this is the only possibility in \cite[Remark 2.5]{Mon13b}. This implies our claim.
\end{itemize}

\end{proof}
\end{prop}

Recently, Bakker \cite[Theorem 2]{Bak13} has obtained through different means that the class of a line inside a lagrangian $\mathbb{P}^n$ sitting in a manifold of \kntiposp has square $-\frac{n+3}{2}$ and the dual divisor $T$ has square $-2(n+3)$ and divisibility $2$. 

\begin{cor}
Let $X$ be a manifold of \kntiposp with a lagrangian fibration with a smooth section, $n\geq 3$. Then the involution given by changing sign on smooth fibres is not induced.
\begin{proof}
Since $n\geq 3$, it follows from \Ref{prop}{lagr_lattice} and its proof that the involution acts non-trivially on $A_X$, therefore the invariant lattice inside $\Lambda_{24}$ is isometric to the one on $X$ and does not contain a copy of $U$.
\end{proof}
\end{cor}

\begin{rmk}
Note that this is another class of examples of automorphisms acting non-trivially on the discriminant group.
\end{rmk}

\begin{prop}
Keep notation as above and let $X$ be deformation equivalent to $Og_{10}$. Then the lattice spanned by $T$ and $D$ is isometric to $U$.
\begin{proof}
We start by embedding $H^2(X,\mathbb{Z})$ inside $\Lambda_{24}\oplus U$ and we extend the action of the isometry $\alpha$ to this lattice. We claim that $div(T)=1$. Indeed, suppose $div(T)\neq 1$, \ie $div(T)=3$. This implies that the lattice generated by $T$ and $D$ has discriminant $9$. After we embed it in the rank $26$ unimodular lattice, the invariant lattice of this involution will have discriminant divisible by 3 (regardless of the action of the involution on the discriminant group $A_X$), therefore it cannot be the invariant lattice of an involution. 
\end{proof} 
\end{prop}

\begin{prop}
Keep notation as above and let $X$ be a manifold of Kummer $n$-type. Then the lattice spanned by $T$ and $D$ is isometric to $U(2)$ if $n$ is odd and to $(2)\oplus (-2)$ otherwise.
\begin{proof}
As in \Ref{prop}{lagr_lattice}, we reduce to the case where the involution is not trivial on $A_X$, since here the orthogonal to the embedding inside $\Lambda_8$ has square $2n+2$. Suppose $div(T)\neq 1$: then $T^2\equiv 2k+2$ modulo $8$, which implies our claim. We only need to exclude the case $div(T)=1$, which can be done using Yoshioka's \cite[Proposition 1.2]{Yos12} characterization of the ample cone of moduli spaces on abelian surfaces and extending it to arbitrary deformations with \cite[Theorem 1.3]{Mon13b}. From it, no wall divisor has divisibility $1$, hence our claim. 
\end{proof}
\end{prop}

As we have exploited in the previous propositions, the presence of an involution imposes rather strict conditions on a section. Indeed we can use them to prove that there exist lagrangian fibrations which can not be deformed to a lagrangian fibration with section, as the following example shows.

\begin{ex}
Let $S$ be a generic $K3$ with a polarisation $H$ of degree $2$ and let $X:=S^{[10]}$ with Picard lattice generated by $H$ and the exceptional divisor $\delta$. The class $F:=3H-\delta$ has square $0$ and divisibility $3$, moreover it lies at the border of the movable cone of $X$. Therefore, by \cite[Remark 11.4]{BM13}, there exists an \issp manifold $Y$ birational to $X$, were the image of $F$ is the class of $\pi^*(\mathcal{O}(1))$ under a lagrangian fibration $\pi\,:\,Y\rightarrow \mathbb{P}^{10}$. By keeping the class of $F$ algebraic, this fibration deforms in a $20$ dimensional family. However, no such fibration has a section. Indeed, if $T$ is the dual divisor to the class of a line in the section, $F$ and $T$ must generate a $2$ elementary lattice as in \Ref{prop}{lagr_lattice}. But $div(F)=3$ implies that this can not happen. 
\end{ex}

\end{document}